\documentclass[review]{elsarticle}

\usepackage{lineno,hyperref}
\modulolinenumbers[5]

\journal{Computers and Operations Research}

\usepackage{bm}
\usepackage{amssymb}
\usepackage{mathtools}
\usepackage{breqn}
\usepackage{tikz} 
\usepackage{pgfplots}
\usepackage{amsthm}
\usepackage{amsmath}
\usepackage{float}
\restylefloat{table}

\newcommand{\vect}[1]{\boldsymbol{#1}}

\newtheorem{theorem}{Theorem}

\newtheorem{algorithm}{Algorithm}
\DeclareMathOperator*{\argmax}{arg\,max}

\begin{document}

\begin{frontmatter}

\title{Using the WOWA criterion for two-stage decision making problems}

\author{Jaeyoong Lim}
\ead{jae0908@kaist.ac.kr}
\author{Sungsoo Park\corref{mycorrespondingauthor}}
\cortext[mycorrespondingauthor]{Corresponding author}
\ead{sspark@kaist.ac.kr}
\address{ Department of industrial and Systems Engineering, Korea Advanced Institute of Science and Technology, 291 Daehak-ro, Yuseong-gu, Daejeon 34141, Republic of Korea}




\begin{abstract}
The weighted OWA (WOWA) is a function that aggregates a set of values with weights assigned based on the rank and relative importance of each value. The weighted OWA of uncertain objective functions can generalize many of the criteria that is used in decision making under uncertainty. 

In this paper, we apply the WOWA criterion to two-stage decision making problems, and present decomposition algorithms to solve them. The algorithms are applied to location-transportation problem with uncertain demands and computational results are presented.
\end{abstract}


\end{frontmatter}


\section{Introduction}
\label{sec:1}
Two-stage decision making problems are problems that involve uncertainty where some of the decisions can be made after the uncertainty is realized, and it models many of the practical problems that are under uncertainty. In this problem there are two types of decision variables. The first type of decision variables are decided prior to the realization of the uncertainty. After the uncertainty is realized, second type of decision variables are chosen so as to optimize the objective function. 

Two-stage decision making problems have been developed in two different streams based on how the uncertainty is handled. In 
the stochastic approach, probability distributions of the uncertain parameters are assumed to be known, and a solution that yields the best expected outcome is found. For a detailed model and techniques to solve the problem, refer to \cite{birge2011introduction} and \cite{shapiro2014lectures}. On the other hand, robust approach is used for more limited information on uncertainty. An uncertainty set is given by the decision maker and the approach finds a solution that optimizes the worst case for the given set of uncertainty. Interested reader may refer to \cite{bertsimas2011theory}.

One of the drawbacks of the stochastic optimization is that it is unlikely that the full probability distributions of uncertain parameters are known. In many cases, little or no information of uncertain parameters are available, and in such circumstances the method cannot be used. Inaccurate or subjective probabilities may be applied and used, however, such approach may result in solutions that are likely to give bad outcomes. In contrast, robust optimization gives a solution whose objective function value is acceptable over all the scenarios, and can be used in more general situations. However, it cannot utilize any of the stochastic information even if some of the information may be available. Furthermore, since it only focuses on the worst case, the solution may not be Pareto optimal, that is, there may be solutions that provide better results for some scenarios while having no worse result for any scenarios \cite{dubois1999computing}.   
One way to compromise between the two approaches is to use the weighted OWA (WOWA) operator. The WOWA operator is an extension of ordered weighted averaging (OWA) operator, a rank dependent operator developed by Yager \cite{yager1988ordered}. The OWA operator aggregates a given set of values with weights assigned based on their rank. That is, for a weight vector $\vect{w} = (w_1,...,w_K)$ such that $w_k \geq 0$ $\forall k=1,...,K$ and $\sum_{k=1}^K w_k =1$, given a vector $\vect{a}\in \mathbb{R}^K$,
\begin{linenomath*}
\begin{align*}
\text{owa}_{(\vect{w})} (\vect{a}) = \sum_{k=1}^K w_k a_{\tau(k)},
\end{align*}
\end{linenomath*}
where $\tau$ is a permutation of $\{1,...,K\}$ such that $a_{\tau(1)} \geq ... \geq a_{\tau(K)}$. The OWA operator can be used in optimization problems with discrete uncertainty to aggregate uncertain objective functions. If the aggregated function is used as a criterion, it generalizes many of the criteria used in decision making under uncertainty \cite{kasperski2016robust}. For example, in a minimization problem, setting $w_1=1, w_k=0$ $\forall k \neq 1$ gives the minmax criterion while setting $w_1 = \alpha, w_K=1-\alpha,$ and 0 otherwise gives the Hurwicz criterion. When the weights are non-decreasing ($w_1 \geq,...,\geq w_K$), higher weights are assigned to higher solution costs and hence can be used to obtain robust solutions. One of the merits of using this approach is that the soltuions that are not pareto optimal can be avoided by using positive values for the weight vector ($w_k>0$ $\forall k\in \{1,...,K\}$) \cite{kasperski2016robust}. The OWA operator has been used extensively in group decision making, informational retrieval, learning algorithms, etc. \cite{yager2012ordered}. With its applications, numerous methods for determining weights for OWA have been investigated as well \cite{xu2005overview}. Also, many generalized versions of the operator, such as induced ordered weighted averaging (IOWA) and generalized ordered weighted averaging (GOWA) have been proposed and used \cite{yager1999induced}, \cite{yager2004generalized}.

The WOWA operator, proposed by Torra \cite{torra1997weighted}, is a generalization of the OWA operator that incorporates a rank independent weight as well. The operator allows additional information, such as the probability of each scenario, to be accounted in decision making under uncertainty, by using it as a rank independent weight. To differentiate the rank dependent weight vector from the rank independent weight vector, we call the former one \textit{preferential vector} and the latter one \textit{importance vector}. Let $\vect{w}$ and $\vect{p}$ be a preferential vector and an importance vector of dimension $K$ respectively such that $w_k \in [0,1]$, $\sum_k w_k = 1$, and $p_k \in [0,1]$, $\sum_k p_k = 1$. Given a vector $\vect{a}=(a_1,...,a_K)$, the weighted OWA operator applied to $\vect{a}$ is defined as:
\begin{linenomath*}
\begin{align*}
\text{wowa}_{(\vect{w}, \vect{p})} (\vect{a}) = \sum_{k=1}^K \delta_k a_{\tau(k)},
\end{align*}
\end{linenomath*}
where $\tau$ is a permutation of $\{1,...,K\}$ such that $a_{\tau(1)} \geq ... \geq a_{\tau(K)}$. The weights $\delta$ are defined as $\delta_k = w^*(\sum_{j \leq k} p_{\tau(j)}) - w^*(\sum_{j < k} p_{\tau(j)})$, where $w^*: [0,1] \rightarrow [0,1]$ is a non-decreasing function that interpolates (0,0) with $(i/K, \sum_{j\leq i} w_j)$ for all $i\in \{1,...K\}$. It should be noted that if objective function for each scenario are aggregated using WOWA operator, the resulting function value can become the expected cost or the worst case cost depending on the choice of the preferential weight and importance weight. For example, if $w_k = 1/K$ $\forall k=\{1,...,K\}$ and the probability of scenarios is used as the importance weight, the WOWA aggregation value becomes the expected cost. On the other hand, if $w_1 = 1, w_k=0$ $\forall k\neq 1$ and $p_k = 1/K$ $\forall k=\{1,...,K\}$, the WOWA aggregation value becomes the worst case performance. Hence using the WOWA criterion for optimization problems with uncertain objective functions provides a link between robust optimization and the stochastic optimization \cite{kasperski2016using}. In fact, the WOWA aggregation value can be seen as the expected cost of a solution using the probability distorted by a rank dependent weight vector (see \cite{diecidue2001intuition} and \cite{torra1998some} for further interpretation of the operator). The WOWA criterion has been applied to many fields including multicriteria optimization \cite{ogryczak2011reference}, metadata aggregation problems \cite{nettleton2001processing} \cite{damiani2006wowa}, and robust discrete optimization problems \cite{kasperski2016using}. 

In this paper we use the WOWA criterion to two-stage decision making problems. In section 2, we define the two-stage decision making problem with WOWA criterion. Linear formulation of the problem and algorithms to solve it are described in section 3. The algorithms are then tested on a location transportation problem with WOWA criterion and their performances are reported in section 4.

\section{Problem definition} 
\label{sec:2}
We focus on the following two-stage decision making problem:
\begin{linenomath*}
\begin{align*}
\min \quad & \vect{c}' \vect{x} + \vect{d}' \vect{y}
\\\text{s.t.} \quad & \vect{A} \vect{x}+\vect{B} \vect{y} = \vect{h} \tag{1} \label{two-stage}
\\ & \vect{x} \in \vect{X}, \vect{x} \geq \vect{0}, \vect{y} \geq \vect{0},
\end{align*}
\end{linenomath*}
where $\vect{x}=\{x_1,...,x_{n_1}\}$ are the first-stage decision variables, $\vect{y}=\{y_1,...,y_{n_2}\}$ are the second-stage decision variables and some or all of $(\vect{d}, \vect{A}, \vect{B}, \vect{h})$ are uncertain with $K$ possible realizations: $(\vect{d}_1, \vect{A}_1, \vect{B}_1, \vect{h}_1),..., (\vect{d}_K, \vect{A}_K, \vect{B}_K, \vect{h}_K)$.

\subsection{Two-stage decision making problem with WOWA criterion}
\label{subsec:2.2}
Let
\begin{linenomath*}
\begin{align*}
Q_k(\vect{x}) =\min & \quad \vect{d}_k' \vect{y}_k
\\ \text{s.t.} & \quad \vect{B}_k \vect{y}_k = \vect{h}_k - \vect{A}_k \vect{x} \tag{2} \label{form:q}
\\ & \quad \vect{y}_k \geq \vect{0}, 
\end{align*}
\end{linenomath*}
that is, $Q_k(\vect{x})$ is a recourse function for scenario $k$, with fixed first-stage solution $\vect{x}$. Given a preferential vector $\vect{w}$ and an importance vector $\vect{p}$, the two-stage decision making problem with WOWA criterion can be written as:
\begin{linenomath*}
\begin{align*}
\min \quad & \vect{c}' \vect{x} + \text{wowa}_{(\vect{w}, \vect{p})} (Q_1(\vect{x}),...,Q_K(\vect{x})) \tag{3} \label{form:WOWA_two-stage}
\\\text{s.t.} \quad & \vect{x} \in \vect{X}, \vect{x} \geq \vect{0}.
\end{align*}
\end{linenomath*}
We assume non-increasing preferential weight ($w_1 \geq w_2,...,\geq w_K$), and $w^*$ to be a concave function to assure that higher weights are assigned to higher costs. The model can be used whether or not the probability distribution of the scenarios is available. If the probability is unavailable, equal weight can be assigned to the importance vector $\vect{p}$, which then only utilize the scenario set like in the robust approach. However, unlike the robust optimization, if the probabilities for scenarios are available, the information can be utilized by assigning the probabilities to the importance weights. Hence, the model can be viewed as a robust approach with stochastic information incorporated. However, since the risk-neutral stochastic two-stage problem is a special case of the two-stage problem with WOWA criterion, all the negative complexity results for the risk-neutral stochastic two-stage problem hold for the two-stage problem with WOWA criterion. We now discuss the methods to solve the problem.

\section{Solution approach}
\label{sec:3}
\subsection{Linear formulation}
\label{subsec:3.1}
The two-stage WOWA problem can be formulated as a linear problem using the reformulation method used in \cite{ogryczak2009efficient}. We use a piecewise linear function as the interpolating function $w^*$ for the sake of linear formulation. Given $\vect{x} \in \vect{X}$, let $\tau$ be a permutation such that $Q_{\tau(1)}(\vect{x})\geq Q_{\tau(2)}(\vect{x}) \geq ... \geq Q_{\tau(K)}(\vect{x})$. Let $\widetilde{p}_k = \sum_{j\leq k} p_{\tau(j)}$ for all $k\in [K]$, and $\widetilde{p}_0 = 0$, where $[K]$ denotes $\{1,...,K\}$. Then, 
\begin{linenomath*}
\begin{align*}
\text{wowa}_{(\vect{w},\vect{p})}(Q_1(\vect{x}),...,Q_K(\vect{x})) = K \sum_{j \in [K]} w_j \int_{\frac{j-1}{K}}^{\frac{j}{K}} u_{\vect{x}}(\mu) d\mu \tag{4} \label{equal:1}
\end{align*}
\end{linenomath*}
holds \cite{ogryczak2009efficient}, where $u_{\vect{x}}(\mu) = Q_{\tau(k)}(\vect{x})$ for $\widetilde{p}_{k-1} < \mu \leq \widetilde{p}_k$, $k \in [K]$, $\mu \in (0,1]$. Let $w_{K+1}=0$. Then (\ref{equal:1}) can be written as:
\begin{linenomath*}
\begin{align*}
\text{wowa}_{(\vect{w},\vect{p})}(Q_1(\vect{x}),...,Q_K(\vect{x})) & = K \left( \sum_{j \in [K]} w_j \int_{0}^{\frac{j}{K}} u_{\vect{x}}(\mu) d\mu - \sum_{j \in [K]} w_{j+1} \int_{0}^{\frac{j}{K}} u_{\vect{x}}(\mu) d\mu \right) \\ & = K \sum_{j\in [K]} (w_j - w_{j+1}) \int_{0}^{\frac{j}{K}} u_{\vect{x}}(\mu) d\mu.  \tag{5} \label{equal:2}
\end{align*}
\end{linenomath*}
Note that for a fixed $\vect{x}$, $\int_{0}^{\frac{j}{K}} u_{\vect{x}}(\mu) d\mu$ can be obtained by solving
\begin{linenomath*}
\begin{align*}
\max \quad & \sum_{k \in [K]} z_k Q_k(\vect{x})
\\ \text{s.t.} \quad & \sum_{k \in [K]} z_k = j/K \tag{6} \label{form:int_u}
\\ & 0 \leq z_k \leq p_k, \quad \forall k\in [K].
\end{align*}
\end{linenomath*}
We note that although $Q_k(\vect{x})$ is defined for scenario $k$, $Q_{\tau(1)}(\vect{x}) \geq Q_{\tau(2)}(\vect{x}) \geq ... \geq Q_{\tau(K)}(\vect{x})$ holds, hence the above formulation is correct. Taking the dual of (\ref{form:int_u}) and substituting to (\ref{equal:2}) gives 
\begin{linenomath*}
\begin{align*}
\text{wowa}_{(\vect{w},\vect{p})}(Q_1(\vect{x}),...,Q_K(\vect{x})) = \min \quad &  K \sum_{j\in [K]} (w_j - w_{j+1}) (j\beta_j/K+\sum_{i\in [K]}p_i \alpha_{ij})
\\\text{s.t.} \quad & \beta_j + \alpha_{kj} \geq Q_k(\vect{x}), \quad \forall k \in [K], \forall j \in [K] 
\\ & \alpha_{kj} \geq 0, \quad \forall k\in [K], \forall j\in [K], \tag{7} \label{form:linear_wowa} 
\end{align*}
\end{linenomath*}
and by (\ref{form:q}), the two-stage problem with WOWA criterion given in (\ref{form:WOWA_two-stage}) can be formulated as:
\begin{linenomath*}
\begin{align*}
\min \quad & \vect{c}' \vect{x} + K\sum_{j\in [K]}(w_j-w_{j+1})(j\beta_j/K+\sum_{i\in [K]}p_k \alpha_{kj}) 
\\\text{s.t.} \quad & \vect{A}_k \vect{x}+\vect{B}_k \vect{y}_k = \vect{h}_k, \quad \forall k \in [K] \tag{8a} \label{form:linear}
\\ & \beta_j + \alpha_{kj} \geq \vect{d}'_k \vect{y}_k, \quad \forall k\in[K], \forall j \in[K] \tag{8b} \label{form:linear_coup}
\\ & \vect{y}_k \geq 0, \alpha_{kj}\geq 0, \quad \forall k \in [K], \forall j \in [K], 
\\ & \vect{x} \in X, \vect{x} \geq \vect{0}.
\end{align*}
\end{linenomath*}

\subsection{Decomposition algorithms for two-stage WOWA problem}
One of the techniques frequently used in two-stage stochastic optimization problem is the L-shaped method. The L-shaped method utilizes the L-shaped block structure of the constraints and approximates the nonlinear term (recourse function) in the objective by building an outer linearization \cite{birge2011introduction}. However, the two-stage problem with WOWA criterion does not have a L-shaped block structure due to the coupling constraints (\ref{form:linear_coup}). Hence, the classical L-shaped method cannot be applied. In this section, we present two decomposition algorithms to solve the two-stage WOWA problem.
\subsubsection{Benders decomposition type algorithm}
One possible approach to solve two-stage WOWA problem is to ignore the coupling constraints (\ref{form:linear_coup}) and use Benders decomposition on the remaining L-shaped block structured constraints (\ref{form:linear}).

Consider the dual of (\ref{form:q}):
\begin{linenomath*}
\begin{align*}
\max & \quad \vect{\lambda}'_k(\vect{h}_k-\vect{A}_k \vect{x})	\tag{9} \label{form:q_max}
\\ \text{s.t.} & \quad \vect{\lambda}'_k \vect{B}_k \leq   \vect{d}'_k.
\end{align*}
\end{linenomath*}
Given a $\vect{x} \in \vect{X}$, suppose (\ref{form:q_max}) is unbounded for some $k \in [K]$. Then, by the duality, (\ref{form:q}) is infeasible, implying that no feasible second stage solution exists for first stage solution $\vect{x}$ and scenario $k$. Hence, $\vect{x} \in \vect{X}$ is feasible in (\ref{form:WOWA_two-stage}) if and only if it satisfies
\begin{linenomath*}
\begin{align*}
\vect{\gamma}_k' (\vect{h}_k-\vect{A}_k\vect{x}) \leq 0, \quad  \forall \vect{\gamma}_k \in \vect{\Gamma}_k, \forall k \in [K],
\end{align*}
\end{linenomath*}
where $\vect{\Gamma}_k$ is the set of extreme rays of $\{\vect{\lambda}_k: \vect{\lambda}_k' \vect{B}_k \leq \vect{0}'\}$.  
Also, if (\ref{form:q_max}) has an optimal solution, there exists an optimal solution of (\ref{form:q_max}) that is an extreme point of $\{\vect{\lambda}_k: \vect{\lambda}_k' \vect{B}_k \leq \vect{d}'_k\}$. That is, if $Q_k (\vect{x})$ is bounded, by the strong duality of linear program,
\begin{linenomath*}
\begin{align*}
Q_k (\vect{x}) = \max\{\vect{\lambda}_k' (\vect{h}_k-\vect{A}_k\vect{x}): \vect{\lambda}_k \in \vect{\Lambda}_k\}, 
\end{align*}
\end{linenomath*}
where $\vect{\Lambda}_k$ is the set of extreme points of $\{\vect{\lambda}_k: \vect{\lambda}_k' \vect{B}_k \leq \vect{d}'_k\}$. Hence, using (\ref{form:linear_wowa}) and by introducing auxiliary variables $q_k$ to represent the values of $Q_k (\vect{x})$, the formulation (\ref{form:WOWA_two-stage}) can be expressed as:
\begin{linenomath*}
\begin{align*}
\min \quad & \vect{c}' \vect{x} + K\sum_{j\in [K]}(w_j-w_{j+1})(j\beta_j/K+\sum_{k\in [K]}p_k \alpha_{kj})
\\\text{s.t.} \quad & \beta_j + \alpha_{kj} \geq q_k, \quad \forall k\in[K], \forall j \in[K] 			
\\ \quad & \vect{\gamma}_k' (\vect{h}_k-\vect{A}_k\vect{x}) \leq 0, \quad \forall \vect{\gamma}_k \in \vect{\Gamma}_k,  \forall k\in[K]	 \tag{10} \label{form:WOWA_Lshaped1}
\\ \quad & q_k \geq \vect{\lambda}_k' (\vect{h}_k-\vect{A}_k\vect{x}), \quad  \forall \vect{\lambda}_k \in \vect{\Lambda}_k, \forall k\in[K] 		
\\ \quad & \alpha_{kj} \geq 0, \forall k\in [K], \forall j \in [K]				 		
\\ \quad & \vect{x} \in \vect{X}, \vect{x} \geq 0.
\end{align*}
\end{linenomath*}
The formulation (\ref{form:WOWA_Lshaped1}) can be solved by the following delayed cut generation algorithm:
\begin{algorithm} (Benders decomposition type algorithm)
\\ \noindent \textbf{Step 0.} Set the tolerance level $\epsilon$. Initialize $\vect{\Gamma}'_k$ and $\vect{\Lambda}'_k$ for all $k\in [K]$, where $\vect{\Gamma}'_k \subseteq \vect{\Gamma}_k$ and $\vect{\Lambda}'_k \subseteq \vect{\Lambda}_k$. \\
\textbf{Step 1.} Solve the relaxed master problem of (\ref{form:WOWA_Lshaped1}) with the subset $\vect{\Gamma}'_k$ and $\vect{\Lambda}'_k$. Let $(\vect{x}^*, \vect{q}^*)$ and $LB$ be its optimal solution and optimal value respectively. \\
\textbf{Step 2.} For  $k\in [K]$, solve 
\begin{linenomath*}
\begin{align*}
Q_k(\vect{x}^*) = \max & \quad \vect{\lambda}_k' (\vect{h}_k-\vect{A}_k\vect{x}^*)	
\\ \text{s.t.} & \quad \vect{\lambda}_k' \vect{B}_k \leq \vect{d}'_k.								 		
\end{align*} 
\end{linenomath*}
If $Q_k(\vect{x}^*)$ is unbounded: $\vect{\Gamma}'_k \leftarrow \vect{\Gamma}'_k \cup \vect{\gamma}_k$, where $\vect{\gamma}_k$ is an extreme ray of  $\{\vect{\lambda}_k: \vect{\lambda}_k' \vect{B}_k \leq \vect{0}'\} $ such that $\vect{\gamma}_k'(\vect{h}_k-\vect{A}_k \vect{x}^v)>0$. If a new constraint is added, go to Step 1. Otherwise, go to Step 3.\\
\textbf{Step 3.} Set $UB = \vect{c}' \vect{x}^*+\text{wowa}_{(\vect{w},\vect{p})}(Q_1(\vect{x}^*),...,Q_K(\vect{x}^*))$. If $(UB-LB)/UB < \epsilon$, \textbf{stop}. Otherwise, for all $k\in [K]$, if $q^*_k < Q_k(\vect{x}^*)$, $\vect{\Lambda}'_k \leftarrow \vect{\Lambda}'_k \cup \vect{\lambda}^*_k$, where $\vect{\lambda}^*_k$ is an optimal solution of the problem in Step 2. Go to Step 1.
\end{algorithm}

In short, Algorithm 1 decomposes the L-shaped block structure constraint (\ref{form:linear}) using Benders decomposition and finds a lower approximations of $\vect{d}'_{k} \vect{y}_k$, which is represented by auxiliary variables $q_k$, to be used in (\ref{form:linear_coup}).

\subsubsection{Subgradient based decomposition algorithm}
We now present an alternative decomposition algorithm to solve the two-stage WOWA problem that uses the subgradient method for the optimality cut. In this algorithm, we iteratively obtain the subgradient of the WOWA of the second stage costs to approximate its value while applying Benders feasbility cut if the first-stage solution is infeasible for the second-stage problem.

Let $\theta$ be an auxiliary variable that represents the value of $\text{wowa}_{(\vect{w},\vect{p})}(Q_1(\vect{x}),...,Q_K(\vect{x}))$. Then, (\ref{form:WOWA_two-stage}) can be expressed as:
\begin{linenomath*}
\begin{align*}
\min \quad & \vect{c}' \vect{x} + \theta  
\\\text{s.t.} \quad & \theta \geq \text{wowa}_{(\vect{w}, \vect{p})} (Q_1(\vect{x}),...,Q_K(\vect{x})) 		\tag{11} \label{form:WOWA_Lshaped2}
\\ \quad & \vect{x} \in \vect{X}, \vect{x} \geq 0.
\end{align*}
\end{linenomath*}
If (\ref{form:q_max}) is unbounded, (\ref{form:WOWA_Lshaped2}) is infeasible, and if (\ref{form:q_max}) has an optimal solution, there exists an optimal solution that is an extreme point of the set described by the constraints of (\ref{form:q_max}). Hence, (\ref{form:WOWA_Lshaped2}) is equivalent to:
\begin{linenomath*}
\begin{align*}
\min \quad & \vect{c}' \vect{x} + \theta 
\\\text{s.t.} \quad & \theta \geq \text{wowa}_{(\vect{w}, \vect{p})} (\vect{\lambda}_1' (\vect{h}_1-\vect{A}_1\vect{x}),...,\vect{\lambda}_K' (\vect{h}_K-\vect{A}_K\vect{x})), \quad  \forall \vect{\lambda}_k \in \vect{\Lambda}_k, \forall k\in[K]
\\ \quad & \vect{\gamma}_k' (\vect{h}_k-\vect{A}_k\vect{x}) \leq 0, \quad \forall \vect{\gamma}_k \in \vect{\Gamma}_k, \forall k\in[K]	 \tag{12} \label{form:WOWA_Lshaped2_extended}					 		
\\ \quad & \vect{x} \in \vect{X}, \vect{x} \geq 0,
\end{align*}
\end{linenomath*}
where $\vect{\Gamma}_k$ is the set of extreme rays of $\{\vect{\lambda}_k: \vect{\lambda}_k' \vect{B}_k \leq \vect{0}'\}$, and $\vect{\Lambda}_k$ is the set of extreme points of $\{\vect{\lambda}_k: \vect{\lambda}_k' \vect{B}_k \leq \vect{d}'_k\}$. To solve (\ref{form:WOWA_Lshaped2_extended}), we utilize the following theorem:
\begin{theorem}
If $w^*$ is a concave function,  
\begin{linenomath*}
\begin{align*}
\text{wowa}_{(\vect{w},\vect{p})}(\vect{a}) = \max_{\pi \in \Pi} \sum_{k\in[K]} \{w^*(\sum_{i \leq k} p_{\pi(i)}) - w^*(\sum_{i < k} p_{\pi(i)})\}a_{\pi(k)},
\end{align*}
\end{linenomath*}
where $\Pi$ is the set of all permutations of $[K]$.
\end{theorem}
\begin{proof}
Clearly, $\text{wowa}_{(\vect{w},\vect{p})}(\vect{a}) \leq \max_{\pi \in \Pi} \sum_{k\in[K]} \{w^*(\sum_{i \leq k} p_{\pi(i)}) - w^*(\sum_{i < k} p_{\pi(i)})\}a_{\pi(k)}$ holds because $\text{wowa}_{(\vect{w},\vect{p})}(\vect{a}) = \sum_{k\in[K]} \{w^*(\sum_{i \leq k} p_{\tau(i)}) - w^*(\sum_{i < k} p_{\tau(i)})\}a_{\tau(k)}$, where $\tau$ is a permutation of $[K]$ such that $a_{\tau(1)} \geq ... \geq a_{\tau(K)}$. Suppose $\text{wowa}_{(\vect{w},\vect{p})}(\vect{a}) <  \max_{\pi \in \Pi} \sum_{k\in[K]} \{w^*(\sum_{i \leq k} p_{\pi(i)}) - w^*(\sum_{i < k} p_{\pi(i)})\}a_{\pi(k)}$. Let $\pi^* = argmax_{\pi \in \Pi}  \sum_{k\in[K]} \{w^*(\sum_{i \leq k} p_{\pi(i)}) - w^*(\sum_{i < k} p_{\pi(i)})\}a_{\pi(k)}$. Then $\exists k \in [K]$ such that $a_{\pi^*(k)} < a_{\pi^*(k+1)}$. Let $k'$ be such $k$, and let $\pi'$ be a permutation where $\pi^*(k')$ and $\pi^*(k'+1)$ is interchanged in $\pi^*$. Then 
\begin{linenomath*}
\begin{align*}
& \sum_{k\in[K]} \left\{w^*(\sum_{i \leq k} p_{\pi^*(i)}) - w^*(\sum_{i < k} p_{\pi^*(i)})\right\} a_{\pi^*(k)} - \sum_{k\in[K]} \left\{w^*(\sum_{i \leq k} p_{\pi'(i)}) - w^*(\sum_{i < k} p_{\pi'(i)})\right\}a_{\pi'(k)} \\
 &  = \left[ \left\{w^*(\sum_{i \leq k'} p_{\pi^*(i)}) - w^*(\sum_{i < k'} p_{\pi^*(i)})\right\} a_{\pi^*(k')} + \left\{w^*(\sum_{i \leq k'+1} p_{\pi^*(i)}) - w^*(\sum_{i \leq k'} p_{\pi^*(i)})\right\} a_{\pi^*(k'+1)} \right]\\
 & \quad -  \left[\left\{ w^*(\sum_{i < k'} p_{\pi^*(i)}+p_{\pi^*(k'+1)}) - w^*(\sum_{i < k'} p_{\pi^*(i)})\right\} a_{\pi^*(k'+1)}\right. \\
 & \quad + \left. \left\{ w^*(\sum_{i \leq k'+1} p_{\pi^*(i)}) - w^*(\sum_{i < k'} p_{\pi^*(i)}+p_{\pi^*(k'+1)})\right\} a_{\pi^*(k')} \right]\\
 & = \left\{ w^*(\sum_{i \leq k'} p_{\pi^*(i)}) - w^*(\sum_{i < k'} p_{\pi^*(i)}) - w^*(\sum_{i \leq k'+1} p_{\pi^*(i)}) + w^*(\sum_{i < k'} p_{\pi^*(i)}+p_{\pi^*(k'+1)})\right\} (a_{\pi^*(k')}-a_{\pi^*(k'+1)}) \\ & \leq 0,
\end{align*}
\end{linenomath*}
where the last inequality follows from $w^*$ being a concave function and $\vect{a}_{\pi^*(k)} < \vect{a}_{\pi^*(k+1)}$. Hence, $\pi'$ is also an optimal permutation and thus there exists an optimal permutation that permutes components of $\vect{a}$ in non-increasing order. 
\end{proof}
Recall that $\text{wowa}_{(\vect{w},\vect{p})}(\vect{a}) = \sum_{k=1}^K \delta_k a_{\tau(k)}$, where $\tau$ is a permutation of $\{1,...,K\}$ such that $a_{\tau(1)} \geq ... \geq a_{\tau(K)}$. Hence, in Theorem 1, 
\begin{align*}
\argmax_{\pi \in \Pi} \sum_{k\in[K]} \{w^*(\sum_{i \leq k} p_{\pi(i)}) - w^*(\sum_{i < k} p_{\pi(i)})\}a_{\pi(k)} 
\end{align*}
can be obtained by finding a permutation $\pi$ such that $a_{\pi(1)} \geq ... \geq a_{\pi(K)}$. 

By Theorem 1, (\ref{form:WOWA_Lshaped2_extended}) can be formulated as:
\begin{linenomath*}
\begin{align*}
\min \quad & \vect{c}' \vect{x} + \theta
\\\text{s.t.} \quad & \theta \geq \sum_{k\in[K]} \{w^*(\sum_{i \leq k} p_{\pi(i)}) - w^*(\sum_{i < k} p_{\pi(i)})\}\vect{\lambda}_{\pi(k)}' (\vect{h}_{\pi(k)}-\vect{A}_{\pi(k)}\vect{x}), 
\\ & \qquad \forall (\pi, \vect{\lambda}_1,...,\vect{\lambda}_K) \in \{(\pi, \vect{\lambda}_1,...,\vect{\lambda}_K):\pi \in \Pi, \vect{\lambda}_1 \in \vect{\Lambda}_1,...,\vect{\lambda}_K \in \vect{\Lambda}_K\}  
\\ \quad & \vect{\gamma}_k' (\vect{h}_k-\vect{A}_k\vect{x}) \leq 0, \quad  \forall \vect{\gamma}_k \in \vect{\Gamma}_k, \forall k\in[K]	\tag{13} \label{form:WOWA_Lshaped2_extended3} 
\\ \quad & \vect{x} \in \vect{X}, \vect{x} \geq 0,
\end{align*}
\end{linenomath*}
where $\Pi$ is the set of all permutations of $[K]$. The formulation (\ref{form:WOWA_Lshaped2_extended3}) can be solved by the following delayed cut generation algorithm:
\begin{algorithm} (Subgradient based decomposition algorithm for two-stage WOWA problem)
\\ \noindent \textbf{Step 0.} Set the tolerance level $\epsilon$. Initialize $\vect{\Gamma}'_k$ for all $k\in [K]$ and $\vect{\Lambda}'$, where $\vect{\Gamma}'_k \subseteq \vect{\Gamma}_k$ and $\vect{\Lambda}' \subseteq \{(\pi, \vect{\lambda}_1,...,\vect{\lambda}_K):\pi \in \Pi, \vect{\lambda}_1 \in \vect{\Lambda}_1,...,\vect{\lambda}_K \in \vect{\Lambda}_K\}$. \\
\textbf{Step 1.} Solve relaxed master problem of (\ref{form:WOWA_Lshaped2_extended3}) with the subset $\vect{\Gamma}'_k$ and $\vect{\Lambda}'$. Let $(\vect{x}^*, \theta^*)$ and $LB$ be its optimal solution and optimal value respectively. \\
\textbf{Step 2.} For  $k\in [K]$, solve 
\begin{linenomath*}
\begin{align*}
Q_k(\vect{x}^*) = \max & \quad \vect{\lambda}_k' (\vect{h}_k-\vect{A}_k\vect{x}^*)	
\\ \text{s.t.} & \quad \vect{\lambda}_k' \vect{B}_k \leq \vect{d}_k.								 		
\end{align*} 
\end{linenomath*}
If $Q_k(\vect{x}^*)$ is unbounded: $\vect{\Gamma}'_k \leftarrow \vect{\Gamma}'_k \cup \vect{\gamma}_k$, where $\vect{\gamma}_k$ is an extreme ray of  $\{\vect{\lambda}_k: \vect{\lambda}_k' \vect{B}_k \leq \vect{0}'\} $ such that $\vect{\gamma}_k'(\vect{h}_k-\vect{A}_k \vect{x}^*)>0$. If a new constraint is added, go to Step 1. Otherwise, go to Step 3.\\
\textbf{Step 3.} Set $UB = \vect{c}' \vect{x}^*+\text{wowa}_{(\vect{w},\vect{p})}(Q_1(\vect{x}^*),...,Q_K(\vect{x}^*))$. If $(UB-LB)/UB < \epsilon$, \textbf{stop}. Otherwise, find a permutation $\pi^*$ such that $Q_{\pi^*(1)}(\vect{x}^*) \geq ... \geq Q_{\pi^*(K)}(\vect{x}^*)$, and add 
\begin{linenomath*}
\begin{align*}
\theta \geq \sum_{k\in[K]} \{w^*(\sum_{i \leq k} p_{\pi^*(i)}) - w^*(\sum_{i < k} p_{\pi^*(i)})\}{\vect{\lambda}^{*}}'_{\pi^*(k)} (\vect{h}_{\pi^*(k)}-\vect{A}_{\pi^*(k)}\vect{x})
\end{align*}
\end{linenomath*}
to the constraint set of the relaxed master problem, where $\vect{\lambda}^*_k$ is an optimal solution of the problem for $Q_k(\vect{x}^*)$ in Step 2. Go to Step 1.
\end{algorithm}
In Algorithm 2, for a first stage solution $\vect{x}^* \in \hat{\vect{X}}$, where $\hat{\vect{X}} = \{\vect{x} \in \vect{X}: Q_k(\vect{x})<\infty, \forall k\in[K]\}$, 
\begin{linenomath*}
\begin{align*}
\text{wowa}_{(\vect{w}, \vect{p})} (Q_1(\vect{x}),...,Q_K(\vect{x})) & = \max_{\pi \in \Pi} \sum_{k\in[K]} \{w^*(\sum_{i \leq k} p_{\pi(i)}) - w^*(\sum_{i < k} p_{\pi(i)})\}Q_{\pi(k)}(\vect{x}), 
\\ & \qquad \forall \vect{x} \in \hat{\vect{X}}
\\ & \geq \max_{\pi \in \Pi} \sum_{k\in[K]} \{w^*(\sum_{i \leq k} p_{\pi(i)}) - w^*(\sum_{i < k} p_{\pi(i)})\} 
\\ & \qquad \cdot {\vect{\lambda}^{*}}'_{\pi(k)} (\vect{h}_{\pi(k)}-\vect{A}_{\pi(k)}\vect{x}), \quad \forall \vect{x} \in \hat{\vect{X}}
\\ & \geq  \sum_{k\in[K]} \{w^*(\sum_{i \leq k} p_{\pi^*(i)}) - w^*(\sum_{i < k} p_{\pi^*(i)})\} 
\\ & \qquad \cdot {\vect{\lambda}^{*}}'_{\pi^*(k)} (\vect{h}_{\pi^*(k)}-\vect{A}_{\pi^*(k)}\vect{x}), \quad \forall \vect{x} \in \hat{\vect{X}}
\\ & = \text{wowa}_{(\vect{w}, \vect{p})} (Q_1(\vect{x^*}),...,Q_K(\vect{x^*}))
\\ & \qquad -\sum_{k\in[K]} \{w^*(\sum_{i \leq k} p_{\pi^*(i)}) - w^*(\sum_{i < k} p_{\pi^*(i)})\} 
\\ & \qquad \cdot {\vect{\lambda}^{*}}'_{\pi^*(k)} \vect{A}_{\pi^*(k)}(\vect{x}-\vect{x}^*), \quad \forall \vect{x} \in \hat{\vect{X}}
\end{align*}
\end{linenomath*}
holds by Theorem 1 and the duality. The fact that there exists a linear under-estimate of $\text{wowa}_{(\vect{w}, \vect{p})} (Q_1(\vect{x}),...,Q_K(\vect{x}))$ that supports the function at $\vect{x}^*$ for every $\vect{x}^* \in \hat{\vect{X}}$ implies that the function is convex for the domain $\hat{\vect{X}}$. Therefore, if $\hat{\vect{X}}$ is a convex set, 
\begin{align*}
-\sum_{k\in[K]} \{w^*(\sum_{i \leq k} p_{\pi^*(i)}) - w^*(\sum_{i < k} p_{\pi^*(i)})\} \cdot {\vect{\lambda}^{*}}'_{\pi^*(k)}\vect{A}_{\pi^*(k)}
\end{align*}
is a subgradient of the function $\text{wowa}_{(\vect{w}, \vect{p})} (Q_1(\vect{x}),...,Q_K(\vect{x}))$ at $\vect{x^*}$. Hence, the cut added in Step 3 can be interpreted as a lower approximation of \\$\text{wowa}_{(\vect{w}, \vect{p})} (Q_1(\vect{x}),...,Q_K(\vect{x}))$ that uses a subgradient of $\text{wowa}_{(\vect{w}, \vect{p})} (Q_1(\vect{x}),...,Q_K(\vect{x}))$ at $x^*$.


To summarize, Algorithm 2 uses feasibility cuts from Benders decomposition method for feasibility of the first-stage solution and subgradients of WOWA function for the optimality.

\section{Computational result}
\label{sec:4}
In this section, we compare the performance of algorithms discussed by applying them to the location transportation problem with uncertain demands \cite{gabrel2014robust}. Let $m$ be the number of possible sites where facilities can be built and let $M_i$, for $i \in [m]$ be the capacity for the facility in site $i$. A commodity have to be transported to $n$ sites from the facilities built however the demands are uncertain with number $K$ of possible scenarios. Let $h^k_j$ for $j\in [n], k\in [K]$ denote the demand at site $j$ under scenario $k$, and let $p_k$ denote the probability of scenario $k$ occurring. Let $f_i$, $c_i$ for $i\in[m]$ be the fixed cost and variable cost at site $i$, respectively. Finally, let $d_{ij}$ for $i\in [m], j\in[n]$ be the transportation cost of a single unit of commodity from site $i$ to site $j$. The objective of the problem is to choose the locations for the facilities ($z_i$), the amount of the commodity to be produced at each facility ($x_i$), and the amount to be transported ($y^k_{ij}$) after the demand is revealed so that the total cost is minimized. The risk-neutral stochastic version of the problem, which minimizes the expected total cost, is given as:
\begin{linenomath*}
\begin{align*}
\min \quad & \sum_i f_i z_i + \sum_i c_i x_i + \sum_i \sum_j \sum_k d_{ij} p_k y_{ij}^k
\\ \text{s.t.} \quad & x_i \leq M_i z_i, \quad \forall i \in [m]
\\ & \sum_j y_{ij}^k \leq x_i, \quad \forall i \in [m], \forall k\in [K]
\\ & \sum_i y_{ij}^k \geq h^k_j, \quad \forall j \in [n], \forall k \in [K]
\\ & x_{ij} \geq 0, y_{ij}^k \geq 0, \quad \forall i \in [m], \forall j \in [n], \forall k \in [K]
\\ & z_i \in \{0,1\}, \quad \forall i \in [m].
\end{align*}
\end{linenomath*}
The WOWA version of the problem with a preferential weight $\vect{w}$ and the probability of scenarios as an importance weight $\vect{p}$ can be expressed as:
\begin{linenomath*}
\begin{align*}
\min \quad & \sum_i f_i z_i + \sum_i c_i x_i + \text{wowa}_{(\vect{w}, \vect{p})} (\sum_i \sum_j d_{ij} y_{ij}^1,...,\sum_i \sum_j d_{ij} y_{ij}^K)     
\\ \text{s.t.} \quad & x_i \leq M_i z_i, \quad \forall i \in [m] \tag{14} \label{form:TL_WOWA}
\\ & \sum_j y_{ij}^k \leq x_i, \quad \forall i \in [m], \forall k \in [K]
\\ & \sum_i y_{ij}^k \geq h^k_j, \quad \forall j \in [n], \forall k \in [K]
\\ & x_{ij} \geq 0, y_{ij}^k \geq 0, \quad \forall i \in [m], \forall j \in [n], \forall k \in [K]
\\ & z_i \in \{0,1\}, \quad \forall i \in [m].
\end{align*}
\end{linenomath*}
To solve (\ref{form:TL_WOWA}), we use the linear formulation of Section 3.1, and two decomposition algorithms illustrated in Section 3.2. The demands $h^k_j$ were generated from integers uniformly distributed in [$\bar{h}_j, 2\bar{h}_j$], for all $k\in [K], \forall j \in [n]$, where $\bar{h}_j$ are taken from integers in [$10, 500$]. The fixed costs ($f_i$), variable costs ($c_i$) and transportation costs ($d_{ij}$) were generated from integers with [$100, 500$], [$10, 50$] and [$1, 1000$], respectively for all $i\in [m], j\in[n]$. For the probability of each scenario, we set $p_k = \bar{p}_k/\sum_{k \in K} \bar{p}_k$, for all $k\in [K]$, where $\bar{p}_k$ is a random integer from [$1,100$]. The preferential weight $\vect{w}$ were set as $w_j = g_{0.1} (j/K) - g_{0.1}((j-1)/K))$ for $j \in [K]$, where $g_\alpha (z) = \frac{1}{1-\alpha} (1-\alpha^z)$. The weight generated by $g_\alpha(z)$ gives non-increasing weight for $\alpha \in (0,1)$, as illustrated in Figure 1. Finally, a piecewise linear function was used for the interpolation function $w^*$.

\begin{figure}[h]
\centering
\begin{tikzpicture}
\begin{axis}[
	xlabel = $z$,
	ylabel = $g_{0.1}(z)$,
	xmin=0, xmax=1,
    	ymin=0, ymax=1,
    	axis lines=left,
    	domain=0:1,
    	xmajorgrids = true,
    	ymajorgrids = true,
    	grid style = dashed,
    	]
	
    \addplot [mark=none,draw=black] {(1/0.9)*(1-0.1^x};
    \addplot+ [only marks, mark = *, forget plot, color = black, fill =black]
    		coordinates {(0.2,0.41)(0.4,0.6688)(0.6,0.832)(0.8,0.935)(1,1)};
    		
	\addplot+ [no marks, color = black, densely dashed]
		coordinates {(0.2,0)(0.2,0.41)};
	extra description/.code={
		\node at (25,20.5) {$w_1$};
	},
	\addplot+ [no marks, color = black, densely dotted]
		coordinates {(0.2,0.41)(0.4,0.41)};

	\addplot+ [no marks, color = black, densely dashed]
		coordinates {(0.4,0.41)(0.4,0.6688)};
	extra description/.code={
		\node at (45,53.94) {$w_2$};
	},
	\addplot+ [no marks, color = black, densely dotted]
		coordinates {(0.4,0.6688)(0.6,0.6688)};

	\addplot+ [no marks, color = black, densely dashed]
		coordinates {(0.6,0.6688)(0.6,0.832)};
	extra description/.code={
		\node at (65,75.04) {$w_3$};
	},
	\addplot+ [no marks, color = black, densely dotted]
		coordinates {(0.6,0.832)(0.8,0.832)};

	\addplot+ [no marks, color = black, densely dashed]
		coordinates {(0.8,0.832)(0.8,0.935)};
	extra description/.code={
		\node at (85,88.35) {$w_4$};
	},
	\addplot+ [no marks, color = black, densely dotted]
		coordinates {(0.8,0.935)(1,0.935)};

	\addplot+ [no marks, color = black, densely dashed]
		coordinates {(0.999,0.935)(0.999,1)};
	extra description/.code={
		\node at (95,96.75) {$w_5$};
	},
\end{axis}
\end{tikzpicture}
\caption{Function $g_{0.1}(z)$ and its generated weights for $K=5$.}
\end{figure}
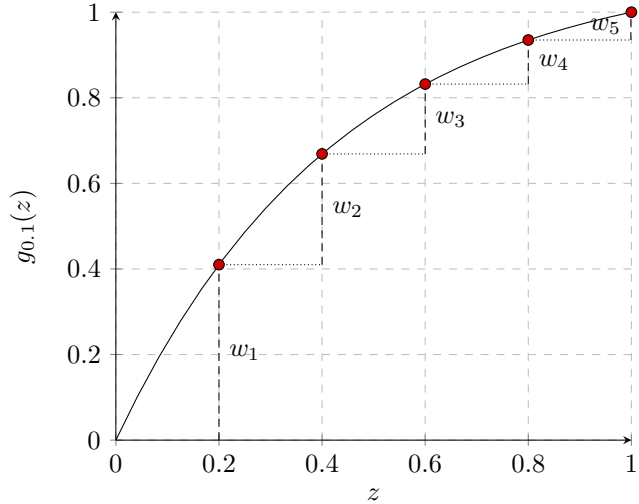

Experiments were conducted with $(n,m) \in \{(30,30), (50,50), (80,80)\}$ and $K \in \{100, 200, 500\}$. For each $(n,m,K)$, 10 instances were generated resulting in total of 90 instances. The linear formulation of the problem, the master problem and the recourse problem for both the Benders decomposition type algorithm and the subgradient based decomposition algorithm were solved with CPLEX 12.8 with default tolerances. For the tolerance level $\epsilon$ in the Benders type decomposition algorithm and the subgradient based decomposition algorithm, $10^{-6}$ was used. The algorithms were run on a computer with a 3.4GHz processor and 8GB RAM.

Table 1 shows the average performance of the 10 instances for each $(n,m,K)$. Each column in the table shows the average computation time, the average number of iterations, the average percent of the time spent on solving the master problem and the subproblem, and the number of instances solved within the time limit of 3600 seconds. Instances that were not solved within the time limit were omitted in computing the average. In the table, we used ``Benders" to denote the algorithm in Section 3.2.1, and ``Subgradient" to denote the algorithm in Section 3.2.2. 
\begin{table}[hp]
\centering
\caption{Results of the algorithms for the location-transportation problem}
\label{table1:result}
\begin{tabular}{ccccccc}
\hline \hline
(n,m,K)                 & Algorithm     & Time(s) & \# of iter. & Master(\%) & Sub(\%) & Solved \\ \hline
(30,30,100)             & Linear form & 8.47    & -        & -          & -       & 10/10  \\
                        & Benders    & 8.44    & 15.2     & 93.1\%     & 6.9\%   & 10/10  \\
                        & Subgradient    & 2.66    & 47.3     & 22.1\%     & 77.9\%  & 10/10  \\    \hline                  
(30,30,200)             & Linear form & 36.70   & -        & -          & -       & 10/10  \\
                        & Benders     & 74.40   & 16       & 98.2\%     & 1.8\%   & 10/10  \\
                        & Subgradient    & 4.80    & 54.8     & 12.7\%     & 87.3\%  & 10/10  \\ \hline
(30,30,500)             & Linear form & 3145.05 & -        & -          & -       & 5/10   \\
                        & Benders     & 1353.56 & 11.4     & 99.9\%     & 0.1\%   & 5/10   \\
                        & Subgradient    & 17.41   & 79.5     & 22.3\%     & 77.7\%  & 10/10  \\\hline
(50,50,100)             & Linear form & 21.57   & -        & -          & -       & 10/10  \\
                        & Benders    & 24.14   & 17.9     & 85.5\%     & 14.5\%  & 10/10  \\
                        & Subgradient    & 14.90   & 82.9     & 21.5\%     & 78.5\%  & 10/10  \\\hline
(50,50,200)             & Linear form & 119.46  & -        & -          & -       & 10/10  \\
                        & Benders     & 274.86  & 20.4     & 95.5\%     & 4.5\%   & 10/10  \\
                        & Subgradient    & 25.79   & 104.7    & 13.8\%     & 86.2\%  & 10/10  \\\hline
(50,50,500)             & Linear form & t.l.    & -        & -          & -       & 0/10   \\
                        & Benders    & 1081.63 & 13       & 99.7\%     & 0.3\%   & 5/10   \\
                        & Subgradient    & 33.66   & 71.4     & 8.5\%      & 91.5\%  & 10/10  \\\hline
(80,80,100)             & Linear form & 65.99   & -        & -          & -       & 10/10  \\
                        & Benders    & 65.67   & 19.4     & 34.5\%     & 65.5\%  & 10/10  \\
                        & Subgradient    & 209.28   & 87.9     & 2\%     & 98\%  & 10/10  \\\hline
(80,80,200)             & Linear form & 750.83  & -        & -          & -       & 10/10  \\
                        & Benders    & 580.08  & 23    & 67.9\%     & 32.1\%   & 8/10  \\
                        & Subgradient   & 615.96  & 135.63    & 1.7\%     & 98.3\%  & 8/10  \\\hline
(80,80,500)             & Linear form & t.l.    & -        & -          & -       & 0/10   \\
                        & Benders     & t.l. & -       & -     & -   & 0/10   \\
                        & Subgradient    & 1270.66   & 113.57     & 0.5\%      & 99.5\%  & 7/10  \\ \hline \hline
\end{tabular}
\end{table}
The results in Table 1 shows that the subgradient based decomposition algorithm outperforms both the linear formulation and the Benders decomposition approach when the number of scenarios is large. For $(n,m,K) = (50,50,500)$, none of the instances were solved within the time limit using the linear formulation, while subgradient based decomposition algorithm solved all the instances in a reasonable time. However, when the number of scenarios is small relative to the size of its nominal problem, the performance of both the linear formulation and the Benders decomposition type algorithm improved compared to the subgradient based decomposition algorithm. For example, when $(n,m,K) = (80,80,100)$, the other two algorithms showed better performance than the subgradient based decomposition algorithm. 

When the two decomposition algorithms are compared, Benders decomposition type algorithm spends most of its computational time in solving the master problem while subgradient based decomposition algorithm spends most of its computational time in solving the subproblems. The master problem of Benders type decomposition contains $K^2$ constraints initially, and hence is larger in size than that of the subgradient based decomposition algorithm. As a trade-off, more iterations are needed to solve the subgradient based method. This results in the difference in computation times spent for solving the master problem and the subproblem for each algorithm.

The preferential weight $w$ can be made to decrease more steeply by taking smaller $\alpha$ values for the generating function $g_\alpha(z)$ as shown in Figure 2. 
\begin{figure}[h]
\centering
\begin{tikzpicture}
\begin{axis}[
	xlabel = $z$,
	ylabel = $g_{\alpha}(z)$,
	xmin=0, xmax=1,
    	ymin=0, ymax=1,
    	axis lines=left,
    	domain=0:1,
    	xmajorgrids = true,
    	ymajorgrids = true,
    	grid style = dashed,
    	]
	\addplot [mark=none,draw=black] {(1/0.9)*(1-0.1^x} node[pos = 0.6, label = below:{$10^{-1}$}]{};
	\addplot [mark=none,draw=black,densely dashed] {(1/0.99)*(1-0.01^x} node[pos = 0.6, label = below:{$10^{-2}$}]{};
	\addplot [mark=none,draw=black,densely dotted] {(1/0.99999)*(1-0.00001^x} node[pos = 0.6, label = below:{$10^{-5}$}]{};
\end{axis}
\end{tikzpicture}
\caption{Function $g_{\alpha}(z)$ for $\alpha \in \{10^{-1}, 10^{-2}, 10^{-5}\}$.}
\end{figure}
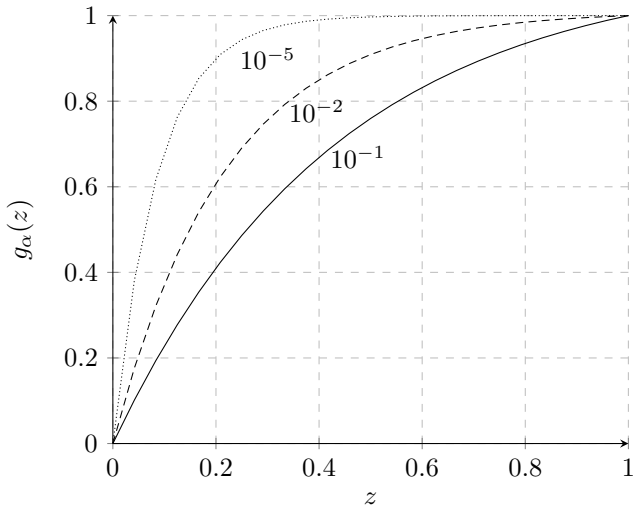    	
To observe the change in the quality of the solutions for different $\alpha$ values, we have compared the expected cost and the cost when the worst scenario is realized for solutions obtained using different $\alpha$ values. That is, given an optimal first stage solution $\vect{x}^*$ obtained from using a given $\alpha$ value, we have compared $\vect{c}' \vect{x}^* + \sum_{k\in[K]} p_k Q_k(\vect{x}^*)$, and $\vect{c}' \vect{x}^* + \max_{k\in[K]} Q_k(\vect{x}^*)$. The results are depicted in Figure 3 as expected cost and worst case cost depending on $\alpha$ values. Since $g_\alpha(z)$ is not defined for $\alpha =1 $, we have set $w_k = 1/K$ for all $k\in [K]$, which leads to the risk-neutral stochastic version of the problem that minimizes the expected cost. Also, for $\alpha = 10^{-10}$, we have set $w_1 = 1, w_k=0$ $\forall k\neq 1$ and $p_k = 1/K$ $\forall k=\{1,...,K\}$ which corresponds to the robust version of the problem that minimizes the cost of the worst case scenario. We can observe that the solutions using WOWA operator gives a solution whose expected cost and the cost for the worst scenario is somewhere between the risk-neutral stochastic version and the robust version, depending on the choice of $\alpha$.    
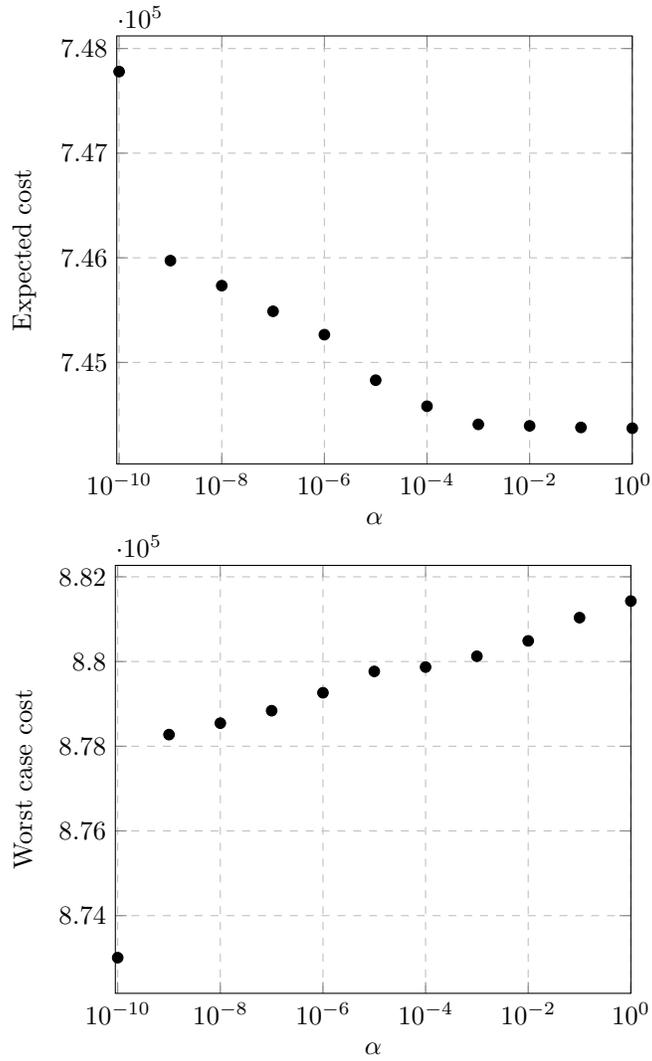
\begin{figure}[h!]
\centering
	\begin{tikzpicture}[scale = 1]
	\begin{axis}[
		xmode=log,
		xlabel = $\alpha$,
		ylabel = Expected cost,
		xmin=0.00000000009, xmax=1.001,
    		xmajorgrids = true,
    		ymajorgrids = true,
    		grid style = dashed,
	]
	\addplot[only marks, color = black] table {1	744371.0406
	0.1	744378.3847
	0.01	744393.7695
	0.001	744407.7515
	0.0001	744581.284
	0.00001	744830.1377
	0.000001	745265.2721
	0.0000001	 745488.4656
	0.00000001	745733.7655
	0.000000001	745972.8966
	0.0000000001	 747779.5787 
	};
	\end{axis}
	\end{tikzpicture}
	\quad
	\begin{tikzpicture}[scale = 1]
	\begin{axis}[
		xmode=log,
		xmin=0.00000000009, xmax=1.001,
		xlabel = $\alpha$,
		ylabel = Worst case cost,
		xmajorgrids = true,
    		ymajorgrids = true,
    		grid style = dashed,
	]
	\addplot [only marks, color = black] table {
	1	881427.6753
	0.1	881035.8638
	0.01	 880487.6357
	0.001	880124.0488
	0.0001	879868.6204
	0.00001	879766.6746
	0.000001	879262.7078
	0.0000001		878839.0163
	0.00000001	878544.6331
	0.000000001	878274.6205
	0.0000000001 873008
	};
	\end{axis}
	\end{tikzpicture}
\caption{Comparison of solutions with different $\alpha$ values}
\end{figure}

\section{Conclusion}
In this paper we applied the weighted OWA (WOWA) criterion to the two-stage decision making problems. Use of the criterion allows  probabilistic information to be accounted in robust approach. We proposed a Benders decomposition type algorithm and a subgradient based decomposition algorithm for the two-stage WOWA problem. The algorithms were tested on a location-transportation problem with WOWA criterion. The computational result shows that the subgradient based decomposition algorithm yields the best performance when the number of scenarios is large, however, for instances with small number of scenarios, both linear formulation and the Benders decomposition type algorithm showed competitive performance. The computational experiment also confirmed that using the WOWA criterion can give solutions that have better worst scenario costs with little loss in the expected costs, and vice versa.

\section*{References}
\bibliography{WOWA_two-stage}

\end{document}